\documentclass[11pt,reqno]{amsart}

\usepackage[letterpaper,margin=1in]{geometry}
\usepackage[T1]{fontenc}
\usepackage[utf8]{inputenc}
\usepackage{lmodern}
\usepackage{microtype}
\usepackage{amsmath,amssymb}
\usepackage{enumitem,aliascnt}
\usepackage{xcolor}
\usepackage[authoryear,round]{natbib}
\usepackage[colorlinks=true,linkcolor=blue!55!black,citecolor=blue!55!black,urlcolor=blue!55!black]{hyperref}
\usepackage[nameinlink,capitalise,noabbrev]{cleveref}
\hypersetup{
  pdftitle={Projection geometry and relaxed quasi-orthogonality for inf-sup stable Galerkin methods},
  pdfauthor={Tsogtgerel Gantumur},
  pdfsubject={Relaxed general quasi-orthogonality for inf-sup stable Galerkin methods},
  pdfkeywords={adaptive finite element methods; general quasi-orthogonality; inf-sup stability; compatible projection chains; block LU factorization},
  pdfdisplaydoctitle=true
}

\setlist[enumerate]{label=(\roman*),leftmargin=2.2em}
\allowdisplaybreaks

\newcommand{\R}{\mathbb R}
\newcommand{\Z}{\mathbb Z}
\newcommand{\Torus}{\mathbb T}
\newcommand{\cL}{\mathcal L}
\newcommand{\cG}{\mathcal G}
\newcommand{\Id}{\mathbf{1}}
\newcommand{\ran}{\operatorname{ran}}
\newcommand{\sgn}{\operatorname{sgn}}
\newcommand{\spann}{\operatorname{span}}
\newcommand{\ip}[2]{\left(#1,#2\right)}
\newcommand{\norm}[1]{\lVert #1\rVert}

\newcommand{\abs}[1]{\lvert #1\rvert}
\newcommand{\gqo}{\mathrm{gqo}}
\newcommand{\qo}{\mathrm{qo}}

\theoremstyle{plain}
\newtheorem{theorem}{Theorem}[section]
\newaliascnt{proposition}{theorem}
\newtheorem{proposition}[proposition]{Proposition}
\aliascntresetthe{proposition}
\newaliascnt{lemma}{theorem}
\newtheorem{lemma}[lemma]{Lemma}
\aliascntresetthe{lemma}
\newaliascnt{corollary}{theorem}
\newtheorem{corollary}[corollary]{Corollary}
\aliascntresetthe{corollary}

\theoremstyle{definition}
\newaliascnt{definition}{theorem}
\newtheorem{definition}[definition]{Definition}
\aliascntresetthe{definition}

\theoremstyle{remark}
\newaliascnt{remark}{theorem}
\newtheorem{remark}[remark]{Remark}
\aliascntresetthe{remark}

\title[Projection geometry and relaxed quasi-orthogonality]
{Projection geometry and relaxed quasi-orthogonality\\
for inf-sup stable Galerkin methods}

\author{Tsogtgerel Gantumur}

\address{McGill University, Montr\'{e}al, QC, Canada}
\address{National University of Mongolia, Ulaanbaatar, Mongolia}
\address{Institute of Mathematics and Digital Technology, Mongolian Academy of Sciences, Ulaanbaatar, Mongolia}

\email{gantumur.tsogtgerel@mcgill.ca}


\begin{document}

\begin{abstract}
We give an elementary, coordinate-free proof that uniformly inf-sup stable 
nested Petrov--Galerkin methods on Hilbert spaces satisfy relaxed general 
quasi-orthogonality. More precisely, for any window of $N$ consecutive levels, 
the sum of squared Galerkin increments is bounded by the 
squared error at the beginning of the window times $N^\sigma$, 
where $\sigma<1$. 
The exponent $\sigma$ and the prefactor can be chosen explicitly in terms of the uniform projection bound.

The proof uses Hilbert-space projection estimates, dyadic iteration, and duality, without relying on matrix factorizations. As a corollary, we obtain explicit growth bounds for finite block-$LU$ factors. In the symmetric indefinite case, the positive and negative spectral splittings of the Galerkin detail spaces yield a characterization of full quasi-orthogonality in terms of uniform boundedness of the associated finite-window sign operators.

We also construct a self-adjoint involution and a nested, 
uniformly inf-sup stable Galerkin sequence for which full general 
quasi-orthogonality fails. For every $0<\alpha<1$, the construction 
yields a hierarchy whose finite-window constants have growth of exact order $N^\alpha$, 
along with an explicit target whose Galerkin increments are not square summable.
\end{abstract}

\maketitle

\section{Introduction}

For linear variational problems governed by a symmetric coercive bilinear
form, nested Galerkin approximations \(u_\ell\) satisfy the Pythagoras identity
\[
  \norm{u-u_{\ell+1}}^2
  +
  \norm{u_{\ell+1}-u_\ell}^2
  =
  \norm{u-u_\ell}^2
\]
in the associated energy norm.
This orthogonality mechanism played a central role
in the development of convergence and rate-optimality
theory for adaptive finite element methods; see, for example,
\citet{Doerfler1996,MorinNochettoSiebert2000,Stevenson2007}.
Its telescoping consequence controls the cumulative size of the successive
Galerkin increments.  For indefinite, saddle-point, or nonsymmetric problems,
the exact identity is generally unavailable.  In perturbative settings it can
sometimes be replaced by problem-specific quasi-orthogonality arguments, but
for strongly indefinite or nonsymmetric problems a different substitute
is needed.

General quasi-orthogonality was introduced into the axiomatic theory of
adaptivity to control cumulative Galerkin increments; see \citet{CFPP2014}.
We distinguish the variants needed below. In this paper, \emph{full general
quasi-orthogonality} (full GQO) denotes the estimate
\begin{equation}\label{eq:intro-full-gqo}
 \sum_{k=\ell}^{\infty}\norm{u_{k+1}-u_k}^2
 \le C_{\gqo}\norm{u-u_\ell}^2
 \qquad(\ell\ge0) ,
\end{equation}
where the norm is that of the ambient Hilbert space.
The constant \(C_{\gqo}\) may depend on the bilinear form and the
nested sequence of trial and test spaces, but not on the target \(u\)
or the starting level \(\ell\).
This is stronger than the estimator-controlled form of axiom (A3) in
\citet{CFPP2014}, which also allows a small subtractive perturbation.
In two contemporaneous works, \citet{FeischlFEMBEM} treated the
nonsymmetric Johnson--N\'ed\'elec FEM--BEM coupling, while
\citet{Feischl2019} treated the indefinite Taylor--Hood discretization of
the Stokes system. Both analyses establish rate optimality under
additional mesh-grading assumptions, using problem-specific wavelet-type
Riesz bases, off-diagonal decay, and infinite-dimensional
\(LU\)-factorization. Boundedness of the exact upper triangular factor
and its inverse implies \eqref{eq:intro-full-gqo}: the factor transforms
Galerkin projections into ordinary coordinate truncations; see
\citet[Lemma~3.3]{Feischl2019}. The concrete applications instead use
approximate factorizations to obtain an error-perturbed GQO estimate,
allowing a small multiple of the accumulated squared errors to be
subtracted from the left-hand side; see
\citet[Theorem~7.3]{FeischlFEMBEM} and
\citet[Theorems~3.2 and~5.4]{Feischl2019}.

A different weakening, established by \citet{Feischl2022}, is
\emph{relaxed general quasi-orthogonality}:
\begin{equation}\label{eq:intro-relaxed}
 \sum_{k=\ell}^{\ell+N-1}\norm{u_{k+1}-u_k}^2
 \le C_{\qo}N^{1-\delta}\norm{u-u_\ell}^2,
\end{equation}
for all starting levels \(\ell\ge0\) and window lengths \(N\ge1\),
with constants \(C_{\qo}>0\) and \(\delta>0\) independent of
\(u\), \(\ell\), and \(N\).
Under the standard estimator hypotheses, this sublinear finite-window
bound suffices for linear estimator convergence.
Uniform inf-sup stability was shown there to imply
\eqref{eq:intro-relaxed} for general nested Petrov--Galerkin methods.
The proof passes to finite hierarchical matrices and derives a sub-square-root growth
bound for their $LU$-factors by triangular truncation and Schatten-norm
interpolation.  This removes quasi-orthogonality as a separate assumption in a
large class of adaptive methods.

The purpose of the present paper is to identify the Hilbert-space geometry
behind this implication.  Rather than studying stiffness matrices and their
factorizations, we work with the nested Galerkin projections, whose norms
are uniformly bounded by inf-sup stability, and establish a finite-window
square-function estimate for them with an explicit sublinear exponent
depending only on this bound.

Let
\((P_j)_{j\ge0}\) be a uniformly bounded family of mutually compatible
projections associated with increasing approximation spaces in a Hilbert
space; the precise definition is given in
\Cref{sec:projection-geometry}.  Write
\[
  K:=\sup_{j\ge0}\norm{P_j},
  \qquad
  E_j:=\ran(P_j-P_{j-1}),
  \qquad
  P_{-1}:=0.
\]
Our main results are as follows.

\begin{enumerate}
\item
We give a direct finite-window proof of
power-type square-function estimates: for every consecutive interval \(I\) of
\(n\) indices and every family \(y_j\in E_j\), \(j\in I\), we have
\[
  \frac{1}{\chi^2(1+\vartheta)n^{\sigma}}
  \sum_{j\in I}\norm{y_j}^2
  \le
  \big\|\sum_{j\in I}y_j\big\|^2
  \le
  (1+\vartheta)n^{\sigma}
  \sum_{j\in I}\norm{y_j}^2,
\]
where 
\[
  \vartheta=\sqrt{1-K^{-2}},
  \qquad
  \sigma=\log_2(1+\vartheta) < 1,
  \qquad
  \chi=\min\{2K,K^2\}.
\]
This also recovers the finite
block-\(LU\) growth theorem of \citet[Theorem~18]{Feischl2022}, with the
admissible choice \(2/p=\sigma\) for \(K=C/\gamma\), where \(C\) and
\(\gamma\) are the matrix norm and principal-section stability bounds.
The cited theorem asserts the existence of \(p>2\) without giving a
closed formula.

\item
We apply this projection theorem to nested uniformly inf-sup stable
Petrov--Galerkin methods.  If the bilinear form has continuity constant
\(C_a\) and the discrete inf-sup constant is \(\gamma\), then the associated
Galerkin projections satisfy
\(K\le {C_a}/{\gamma}\).
Taking \(K=C_a/\gamma\) in (i) gives \(\delta=1-\sigma\) in
\eqref{eq:intro-relaxed}.
Consequently, the relaxed exponent and its prefactor are uniform over all
target solutions and all nested sequences sharing these structural constants.
The guaranteed gain can nevertheless be small for large \(K\); see \Cref{cor:sigma-univ}.

\item
In the self-adjoint indefinite case, we split each Galerkin detail space
into its positive and negative spectral parts. The resulting finite-window
sign operator is a block multiplier whose norm grows at most like
\(N^\sigma\). Uniform boundedness of these sign operators is equivalent
to full GQO. If the discrete union is dense, full GQO is further equivalent
to the existence of an equivalent Hilbert norm in which all Galerkin
projections are orthogonal; see \Cref{sec:sign}. This last equivalence is
a Galerkin form of the classical Hilbert-space square-function criterion
for decompositions; see, e.g.,
\citet[Chapter~VI, \S5]{GohbergKrein1969}.

\item We show that uniform inf-sup stability does not imply full GQO.  Using the trigonometric Schauder
basis of the power-weighted Hilbert space
$L^2(\Torus,|t|^\alpha\,dt)$, we construct a self-adjoint involution and a
uniformly stable nested Galerkin sequence for which
\eqref{eq:intro-full-gqo} fails for an explicit target.  For each fixed
\(0<\alpha<1\), the finite-window constants are of order \(N^\alpha\).
\end{enumerate}

As an adaptive consequence, the relaxed quasi-orthogonality estimate obtained
above can be used in the abstract argument of
\citet{Feischl2022}.  Under the standard axioms of
adaptivity, it therefore yields linear estimator convergence and the usual
rate-optimality conclusions; see \S\ref{sec:afem}.

Power-type square-function bounds of this kind are implicit in the
classical theory of bases in superreflexive spaces
\citep{GurariiGurarii1971,James1972}, which yields them only
qualitatively.  The projection-angle argument below proves them in
Hilbert space, for detail spaces of any dimension, with exponent and
constants depending explicitly on \(K\).
The weighted-Fourier counterexample adapts the power-weighted
construction of \citet{GarrigosWojtaszczyk2014} to a fixed uniformly stable
self-adjoint Galerkin hierarchy.

The paper first develops the projection-angle and square-function estimates
and applies them to inf-sup stable Petrov--Galerkin methods and adaptivity.
It then treats the self-adjoint sign geometry and the counterexample to full
GQO, and ends with open questions.

\section{Projection angles and finite-window square functions}
\label{sec:projection-geometry}

The quasi-orthogonality argument depends on a Galerkin discretization only
through the geometry of its solution projections.  For nested trial and test
spaces, the corresponding Galerkin projections form a compatible projection
chain, while uniform discrete inf-sup stability provides a common bound for
their operator norms.  These facts will be verified in
\Cref{sec:galerkin-setting}.  
We begin with the Hilbert-space part of the argument.

Let \(H\) be a real Hilbert space, and let
\((P_j)_{j\ge0}\subset\cL(H)\) be a compatible projection chain, meaning that
each \(P_j\) is a projection and
\begin{equation}
\label{eq:abstract-nest}
  P_jP_k=P_{\min\{j,k\}}
  \qquad (j,k\ge0).
\end{equation}
In particular, the projections commute and their ranges form an increasing
family.
Set
\[
  P_{-1}:=0,
  \qquad
  \Delta_j:=P_j-P_{j-1},
  \qquad
  E_j:=\ran\Delta_j.
\]
The operators \(\Delta_j\) are the associated detail projections.  
The compatibility relation implies
\[
  \Delta_i\Delta_j=\delta_{ij}\Delta_j.
\]
Consequently, for \(0\le r\le m\), we have
\[
  \ran(P_m-P_{r-1})
  =
  \bigoplus_{j=r}^m E_j.
\]
We assume
that the projection chain is uniformly bounded:
\begin{equation}
\label{eq:abstract-K}
  K:=\sup_{j\ge0}\norm{P_j}<\infty.
\end{equation}
Individual \(P_j\) may vanish, and consecutive ones may coincide. 
The
corresponding detail spaces \(E_j\) are then zero, contribute nothing
to the sums below, and are ignored in statements about the \(E_j\).
We exclude only the trivial chain with \(P_j=0\) for every \(j\), so
that \(K\ge1\).
For a nontrivial Hilbert-space projection \(P\), one has
\(\norm{\Id-P}=\norm{P}\): relative to
\(\ran P\oplus(\ran P)^\perp\), its matrix is
\(\left(\begin{smallmatrix}\Id&B\\0&0\end{smallmatrix}\right)\), and both
squared norms equal \(1+\norm{B}^2\).  
On the other hand, the trivial cases \(P=0\) and \(P=\Id\) would have \(\norm{\Id-P}\le1\le K\).  
Since
\(\Delta_j=P_j(\Id-P_{j-1})\), we obtain
\begin{equation}\label{eq:detail-bound}
  \norm{\Delta_j}\le \chi:=\min\{2K,K^2\}
  \qquad(j\ge0).
\end{equation}
The constant $\chi$ and the constants $\vartheta_K,\sigma_K,A_K,B_K$ introduced below are nondecreasing in $K$, so every estimate remains valid if $K$ is replaced by any upper bound for $\sup_j\|P_j\|$.

\subsection{The angle between consecutive blocks}

The uniform bound \(K\) controls
the angle between an initial block of details and the consecutive block that
follows it.  Define
\begin{equation}
  \vartheta_K:=\sqrt{1-K^{-2}}.
\end{equation}
Since \(K\ge1\), we have \(0\le\vartheta_K<1\).  Geometrically,
\(\vartheta_K\) is the largest possible cosine of the angle between two
consecutive detail blocks under the bound \(K\).  When \(K=1\), every \(P_j\)
is an orthogonal projection, so \(\vartheta_K=0\) and consecutive blocks are
orthogonal.

\begin{lemma}[Consecutive-block angle]
\label{lem:block-angle}
Let \(0\le m<s\), and set
\[
  U:=\ran P_m=\bigoplus_{j=0}^m E_j,
  \qquad
  V:=\ran(P_s-P_m)=\bigoplus_{j=m+1}^s E_j.
\]
Then we have
\begin{equation}
\label{eq:block-angle}
  |(u,v)_H|
  \le
  \vartheta_K\norm{u}_H\norm{v}_H
  \qquad (u\in U,\ v\in V),
\end{equation}
and consequently
\begin{equation}
\label{eq:two-block-synthesis}
  \norm{u+v}_H^2
  \le
  (1+\vartheta_K)
  \bigl(\norm{u}_H^2+\norm{v}_H^2\bigr).
\end{equation}
\end{lemma}

\begin{proof}
By \eqref{eq:abstract-nest}, \(P_m\) acts as the identity on \(U\) and
vanishes on \(V\).  Hence \(P_m|_{U\oplus V}\) is the projection onto \(U\)
along \(V\), with norm at most \(K\).  Thus, for every \(t\in\R\), we have
\[
  \norm{u}_H
  =
  \norm{P_m(u-tv)}_H
  \le
  K\norm{u-tv}_H.
\]
For $v\ne0$, minimize the right-hand side over $t$.  Since
\[
 \inf_{t\in\R}\norm{u-tv}_H^2
 =\norm{u}_H^2-\frac{\abs{\ip{u}{v}_H}^2}{\norm{v}_H^2},
\]
we obtain \eqref{eq:block-angle}.  The case $v=0$ is trivial.  
Finally, we infer
\[
 \norm{u+v}_H^2
 \le\norm{u}_H^2+\norm{v}_H^2
   +2\vartheta_K\norm{u}_H\norm{v}_H,
\]
and $2ab\le a^2+b^2$ gives \eqref{eq:two-block-synthesis}.
\end{proof}

\subsection{A dyadic upper square-function estimate}

The preceding angle estimate can be iterated over a dyadic splitting of a
finite interval.  Define
\begin{equation}
  \sigma_K:=\log_2(1+\vartheta_K).
\end{equation}
Since \(0\le\vartheta_K<1\), we have \(0\le\sigma_K<1\).

\begin{theorem}[Upper square-function estimate]
\label{thm:upper-square}
Let $I=\{r,r+1,\ldots,r+n-1\}$ be a consecutive interval of $n\ge1$
indices.  For arbitrary $y_j\in E_j$, $j\in I$, one has
\begin{equation}\label{eq:upper-square}
\big\|\sum_{j\in I}y_j\big\|_H^2
 \le (1+\vartheta_K) n^{\sigma_K}\sum_{j\in I}\norm{y_j}_H^2 .
\end{equation}
\end{theorem}

\begin{proof}
Let \(A(n)\) be the least constant in
\(\norm{\sum_{j\in I}y_j}_H^2\le A(n)\sum_{j\in I}\norm{y_j}_H^2\),
uniformly over intervals of length \(n\) and vectors \(y_j\in E_j\).
The Cauchy--Schwarz inequality gives \(A(n)\le n\), and \(A(1)=1\).  Split an interval of length
$n\ge2$ into consecutive left and right halves of lengths
$n_-:=\lfloor n/2\rfloor$ and $n_+:=\lceil n/2\rceil$.
For an interval starting at \(r\), set \(m=r+n_--1\).
The corresponding left and right partial sums satisfy
\(u\in\ran P_m\) and \(v\in\ran(P_{r+n-1}-P_m)\), so
\Cref{lem:block-angle} gives
\[
 \norm{u+v}_H^2
 \le(1+\vartheta_K)(\norm{u}_H^2+\norm{v}_H^2).
\]
Hence we have
\[
 A(n)\le(1+\vartheta_K)\max\{A(n_-),A(n_+)\}.
\]
Iteration through at most $\lceil\log_2n\rceil$ generations yields
\[
 A(n)\le(1+\vartheta_K)^{\lceil\log_2n\rceil}
 \le(1+\vartheta_K)n^{\log_2(1+\vartheta_K)}.\qedhere
\]
\end{proof}

\subsection{Duality and the reverse square-function estimate}

The reverse estimate follows by applying the upper square-function bound to the
adjoint projection chain and exploiting the resulting biorthogonality identity.

\begin{corollary}[Two-sided square-function estimate]\label{thm:two-sided-square}
Under \eqref{eq:abstract-nest}--\eqref{eq:abstract-K}, let $I$ be a
consecutive interval of $n\ge1$ indices.  Then every family
$y_j\in E_j$, $j\in I$, satisfies
\begin{equation}\label{eq:two-sided-square}
 \frac{1}{B_Kn^{\sigma_K}}
 \sum_{j\in I}\norm{y_j}_H^2
 \le
 \big\|\sum_{j\in I}y_j\big\|_H^2
 \le
 A_Kn^{\sigma_K}\sum_{j\in I}\norm{y_j}_H^2,
\end{equation}
where
$ A_K=1+\vartheta_K$
and
$B_K=\chi^2A_K=\min\{4K^2,K^4\}(1+\vartheta_K)$.
\end{corollary}

\begin{proof}
Only the first inequality remains to be proved.  Put
\[
 w:=\sum_{j\in I}y_j,
 \qquad
 z_j:=\Delta_j^*y_j,
 \qquad
 z:=\sum_{j\in I}z_j.
\]
Since $\Delta_j\Delta_i=\delta_{ij}\Delta_i$, we have
\begin{equation}\label{eq:biorth-identity}
 \ip{w}{z}_H
 =\sum_{i,j\in I}\ip{\Delta_jy_i}{y_j}_H
 =\sum_{j\in I}\norm{y_j}_H^2.
\end{equation}
The adjoint projections \((P_j^*)\) form a compatible projection chain with
the same uniform bound \(K\).  Since \(z_j\in\ran\Delta_j^*\),
\Cref{thm:upper-square} gives
\[
  \norm{z}_H^2
  \le
  A_K n^{\sigma_K}
  \sum_{j\in I}\norm{z_j}_H^2 
  \le \chi^2A_Kn^{\sigma_K}\sum_{j\in I}\norm{y_j}_H^2,
\]
where we have used \eqref{eq:detail-bound}.
Combining this with \eqref{eq:biorth-identity} and the Cauchy--Schwarz inequality, we conclude
\[
 \sum_{j\in I}\norm{y_j}_H^2
 \le \chi\sqrt{A_K}\,n^{\sigma_K/2}\norm{w}_H
 \Big(\sum_{j\in I}\norm{y_j}_H^2\Big)^{1/2}.
\]
Then squaring proves the first inequality in \eqref{eq:two-sided-square}.
\end{proof}

\begin{remark}[Complex Hilbert spaces]
\label{rem:complex}
The square-function estimates above hold on complex Hilbert spaces with
the same constants. Indeed, regard \(H\) as a real Hilbert space with
inner product \(\operatorname{Re}(\cdot,\cdot)_H\).
This preserves the Hilbert norm and, for complex-linear operators,
the operator norms and adjoints. The compatibility relations are also
unchanged, so the preceding proofs apply without modification.
\end{remark}

\subsection{A finite block-\texorpdfstring{\(LU\)}{LU} consequence}

We next record a finite-dimensional consequence of the square-function
estimate.  Throughout this subsection, \(\R^n\) is equipped with its Euclidean
inner product, and \(\norm{\cdot}\) denotes the induced operator norm.  Fix an
orthogonal decomposition
\[
  \R^n=H_1\oplus\cdots\oplus H_m,
\]
and let \(F_j\) be the orthogonal projection onto \(H_j\).  For
\(A\in\R^{n\times n}\), its \((i,j)\)-block is the operator
\[
  A_{ij}:=F_iA|_{H_j}:H_j\to H_i.
\]
We call \(A\) block upper triangular if \(A_{ij}=0\) whenever \(i>j\),
and block lower triangular if \(A_{ij}=0\) whenever \(i<j\).  A normalized
block-\(LU\) factorization \(M=LU\) is one in which \(L\) is block lower
triangular with \(L_{jj}=\Id_{H_j}\)
and \(U\) is block upper triangular.

Set
\(Q_0:=0\),
\(\mathcal H_0:=\{0\}\) and
\[
  Q_j:=\sum_{i=1}^jF_i,
  \qquad
  \mathcal H_j:=\ran Q_j
  =
  H_1\oplus\cdots\oplus H_j
  \qquad\text{for}\quad
  1\le j\le m .
\]
Thus \(Q_j\) is the orthogonal projector onto the partial sum
\(\mathcal H_j\). Note that \(Q_m=\Id\) and \(\mathcal H_m=\R^n\).
The following result recovers \citet[Theorem~18]{Feischl2022}, 
with an explicit exponent.

\begin{corollary}[Finite block-\(LU\) growth]
\label{cor:finite-block-lu}
Let \(M\in\R^{n\times n}\), and for \(1\le j\le m\) define its \(j\)-th
leading principal block section by
\[
  M_j
  :=
  Q_jM|_{\mathcal H_j}
  :
  \mathcal H_j\to\mathcal H_j.
\]
Suppose that
\[
  \norm{M}\le C,
  \qquad
  \|{M_j^{-1}}\|\le\gamma^{-1} ,
\]
for all \(1\le j\le m\) and
for some \(C,\gamma>0\).  
Then we have \(\gamma\le C\), and
\(M\) admits a unique normalized block-\(LU\) factorization \(M=LU\).
With \(K=C/\gamma\), 
\(\chi=\min\{2K,K^2\}\), 
\(A_K=1+\sqrt{1-K^{-2}}\), 
and \(1/p=\log_2(A_K)/2\) so that $p\in(2,\infty]$,
its factors satisfy the scale-invariant bounds
\begin{equation}
\label{eq:finite-lu-growth}
  \max\left\{
    \gamma\norm{U^{-1}},
    \frac{\norm{U}}{\chi C},
    \frac{\norm{L}}{K},
    \frac{\norm{L^{-1}}}{\chi K}
  \right\}
  \le \sqrt{A_K}\,m^{1/p}.
\end{equation}
In particular, the usual unnormalized sum of the four norms is
\(O_{C,\gamma}(m^{1/p})\).
\end{corollary}

\begin{proof}
Since \(M_m=M\), the assumptions imply \(\gamma\le C\).  Set
\(K=C/\gamma\), so that \(\sigma_K=2/p\), and write
\(A_K=1+\vartheta_K\) and \(B_K=\chi^2A_K\).
Invertibility of the leading principal block sections guarantees the
existence and uniqueness of the normalized factorization \(M=LU\).

For \(1\le j\le m\), define
\(P_j:=\iota_jM_j^{-1}Q_jM\), where
\(\iota_j:\mathcal H_j\hookrightarrow\R^n\) is the inclusion, and set
\(P_0:=0\).  
Let \(L_j\) and \(U_j\) be the leading sections of \(L\) and \(U\), respectively.
Then we have \(Q_jL=L_jQ_j\) as maps into \(\mathcal H_j\), and \(U^{-1}\iota_j=\iota_jU_j^{-1}\).
Since \(M_j=L_jU_j\), block triangularity gives
\begin{equation*}
  P_j
  =\iota_jU_j^{-1}L_j^{-1}Q_jLU
  =\iota_jU_j^{-1}Q_jU
  =U^{-1}Q_jU.
\end{equation*}
Thus \((P_j)_{j=0}^m\) is a compatible projection chain,
\(\norm{P_j}\le\norm{M_j^{-1}}\norm{M}\le K\), and
\begin{equation*}
  \Delta_j:=P_j-P_{j-1}=U^{-1}F_jU.
\end{equation*}
The finite chain may be extended by \(P_j=\Id\) for \(j>m\).

We first estimate \(U^{-1}\).  The last block column of \(L_j^{-1}\)
is the corresponding identity column, so \(L_j^{-1}F_j=F_j\), with
values in \(\mathcal H_j\).  Using
\(M_j^{-1}=U_j^{-1}L_j^{-1}\), we obtain
\[
  \iota_jM_j^{-1}F_j
  =\iota_jU_j^{-1}F_j
  =U^{-1}F_j.
\]
Therefore we have
\begin{equation*}
  \norm{U^{-1}F_j}\le\gamma^{-1}
  \qquad(1\le j\le m).
\end{equation*}
For \(x\in\R^n\), put \(y_j:=\Delta_jU^{-1}x=U^{-1}F_jx\).
Since \(P_m=\Id\), we have \(U^{-1}x=\sum_{j=1}^my_j\) and 
\[
  \norm{U^{-1}x}^2
  \le A_Km^{2/p}\sum_{j=1}^m\norm{y_j}^2
  \le A_K\gamma^{-2}m^{2/p}\sum_{j=1}^m\norm{F_jx}^2
  =A_K\gamma^{-2}m^{2/p}\norm{x}^2,
\]
by \Cref{thm:upper-square}.
We thus infer
\begin{equation}\label{eq:Uinv-growth}
  \norm{U^{-1}}\le\gamma^{-1}\sqrt{A_K}\,m^{1/p}.
\end{equation}

To estimate \(U\), set \(y_j:=\Delta_jx\), so that we have \(x=\sum_{j=1}^my_j\) and \(Uy_j=F_jUx\).
The identity diagonal of \(L\) yields
\[
  F_jMy_j=F_jLUy_j=F_jLF_jUx=F_jUx.
\]
Thus \(\norm{F_jUx}\le C\norm{y_j}\), and the reverse estimate in
\Cref{thm:two-sided-square} gives
\[
  \norm{Ux}^2
  =\sum_{j=1}^m\norm{F_jUx}^2
  \le C^2\sum_{j=1}^m\norm{y_j}^2
  \le C^2B_Km^{2/p}\norm{x}^2,
\]
implying
\begin{equation}\label{eq:U-growth}
  \norm{U}\le C\sqrt{B_K}\,m^{1/p}.
\end{equation}
Finally, we have \(L=MU^{-1}\) and \(L^{-1}=UM^{-1}\), so
\begin{equation}
  \norm{L}\le C\norm{U^{-1}}
  \le K\sqrt{A_K}\,m^{1/p},
  \qquad
  \norm{L^{-1}}\le\gamma^{-1}\norm{U}
  \le K\chi\sqrt{A_K}\,m^{1/p}.
\end{equation}
Together with \eqref{eq:Uinv-growth}--\eqref{eq:U-growth}, these are exactly
the bounds in \eqref{eq:finite-lu-growth}.
\end{proof}

\section{Relaxed quasi-orthogonality for inf-sup stable Galerkin methods}
\label{sec:relaxed-gqo}

\subsection{Nested Petrov--Galerkin projections}
\label{sec:galerkin-setting}

Let \(X\) and \(Y\) be real Hilbert spaces, and let \(a\) be a bounded
bilinear form on \(X\times Y\) with
\begin{equation}\label{eq:continuity}
 \abs{a(x,y)}\le C_a\norm{x}_X\norm{y}_Y
 \qquad(x\in X,\ y\in Y).
\end{equation}
Only the discrete stability condition below enters the
quasi-orthogonality proof.  The argument extends, with the same
constants, to complex spaces and sesquilinear forms; see
\Cref{rem:complex}.

Let
\(X_0\subset X_1\subset\cdots\subset X\)
and
\(Y_0\subset Y_1\subset\cdots\subset Y\)
be finite-dimensional nested spaces, not all zero, with
\(\dim X_\ell=\dim Y_\ell\).
Assume the uniform discrete inf-sup condition
\begin{equation}\label{eq:discrete-infsup}
 \inf_{0\ne x_\ell\in X_\ell}
 \sup_{0\ne y_\ell\in Y_\ell}
 \frac{\abs{a(x_\ell,y_\ell)}}
 {\norm{x_\ell}_X\norm{y_\ell}_Y}
 \ge\gamma>0
 \qquad(\ell\ge0) ,
\end{equation}
with \(0<\gamma\le C_a\).
Define the Galerkin projection $G_\ell:X\to X_\ell$ by
\begin{equation}\label{eq:galerkin-projector}
 a(G_\ell x,y_\ell)=a(x,y_\ell)
 \qquad(y_\ell\in Y_\ell).
\end{equation}
The equality of dimensions and \eqref{eq:discrete-infsup} ensure existence and
uniqueness.

\begin{lemma}[Compatibility and boundedness of the Galerkin projections]
\label{lem:galerkin-nest}
The operators \(G_\ell\in\cL(X)\) are projections onto \(X_\ell\) and satisfy
\begin{equation}
\label{eq:nest-algebra}
  G_jG_k=G_{\min\{j,k\}}
  \qquad (j,k\ge0).
\end{equation}
Moreover, we have
\begin{equation}
\label{eq:G-bound}
  \norm{G_\ell}_{\cL(X)}
  \le \frac{C_a}{\gamma}
  \qquad (\ell\ge0).
\end{equation}
\end{lemma}

\begin{proof}
Let \(x_\ell\in X_\ell\).  
By the defining relation \eqref{eq:galerkin-projector}, we have
\[
  a(G_\ell x_\ell-x_\ell,y_\ell)=0
  \qquad (y_\ell\in Y_\ell).
\]
Since \(G_\ell x_\ell-x_\ell\in X_\ell\), the discrete inf-sup condition
\eqref{eq:discrete-infsup} implies \(G_\ell x_\ell=x_\ell\).  Hence
\(G_\ell\) is a projection onto \(X_\ell\).

Now let \(j\le k\) and \(x\in X\).  Since \(G_jx\in X_j\subset X_k\) and \(G_k\) acts as
the identity on \(X_k\), we have
\(G_kG_jx=G_jx\).
Conversely, for every \(y_j\in Y_j\subset Y_k\), one has
\[
  a(G_jG_kx,y_j)
  =a(G_kx,y_j)
  =a(x,y_j).
\]
Thus \(G_jG_kx\in X_j\) satisfies the defining Galerkin equations for
\(G_jx\).  Uniqueness on \(X_j\), again following from the discrete inf-sup
condition, gives \(G_jG_kx=G_jx\).  This proves
\eqref{eq:nest-algebra}.

Finally, for any \(x\in X\), the discrete inf-sup condition and
\eqref{eq:galerkin-projector} yield
\[
  \gamma\norm{G_\ell x}_X
  \le
  \sup_{0\ne y_\ell\in Y_\ell}
  \frac{\abs{a(G_\ell x,y_\ell)}}{\norm{y_\ell}_Y}
  =
  \sup_{0\ne y_\ell\in Y_\ell}
  \frac{\abs{a(x,y_\ell)}}{\norm{y_\ell}_Y}
  \le
  C_a\norm{x}_X.
\]
Taking the supremum over \(x\ne0\) proves \eqref{eq:G-bound}.
\end{proof}

Thus \((G_\ell)_{\ell\ge0}\) is a compatible projection chain with the admissible uniform projection bound
\(K={C_a}/{\gamma}\).
It therefore satisfies the hypotheses of
\Cref{sec:projection-geometry}.

\subsection{Quantitative relaxed quasi-orthogonality}
\label{sec:quantitative-relaxed-gqo}

We now apply the square-function estimate to the adjoint Galerkin projection
chain.  To formulate the result uniformly over the starting level and the
target, let
\(\cG:=(G_\ell)_{\ell\ge0}\)
and introduce the following constants.

\begin{definition}[Finite-window and full GQO constants]
For \(N\ge1\), define
\begin{equation}
\label{eq:C-window}
  \mathfrak C_{\cG}(N):=
  \sup_{\ell\ge0}
  \sup_{\substack{x\in X\\(\Id-G_\ell)x\ne0}}
  \frac{
    \displaystyle
    \sum_{k=\ell}^{\ell+N-1}
    \norm{(G_{k+1}-G_k)x}_X^2
  }{
    \norm{(\Id-G_\ell)x}_X^2
  }.
\end{equation}
Empty suprema are taken to be zero.  The full GQO constant is
\begin{equation}
\label{eq:C-full}
  \mathfrak C_{\cG}(\infty)
  :=
  \sup_{N\ge1}\mathfrak C_{\cG}(N)
  \in[0,\infty].
\end{equation}
\end{definition}

For the adaptive argument, the relevant relaxed-GQO condition is
\(\mathfrak C_{\cG}(N)=o(N)\), whereas full GQO requires
\(\mathfrak C_{\cG}(\infty)<\infty\).
The definition uses only the projections \(G_\ell\), and we use the
same notation for an arbitrary compatible projection chain.

\begin{theorem}[Quantitative relaxed GQO]
\label{thm:main-gqo}
Assume \eqref{eq:continuity} and \eqref{eq:discrete-infsup}, and let
\((G_\ell)_{\ell\ge0}\) be the associated Galerkin projections.  Set
\begin{equation}
\label{eq:galerkin-theta-sigma}
  K=\frac{C_a}{\gamma},
  \qquad
  \vartheta_K=\sqrt{1-K^{-2}},
  \qquad
  \sigma_K=\log_2(1+\vartheta_K)<1.
\end{equation}
For \(u\in X\) and \(\ell\ge0\), define \(u_\ell:=G_\ell u\).  
Then, for every \(\ell\ge0\) and \(N\ge1\), we have
\begin{equation}
\label{eq:main-gqo}
  \sum_{k=\ell}^{\ell+N-1}
  \norm{u_{k+1}-u_k}_X^2
  \le
  C_{\qo}(K)N^{\sigma_K}\norm{u-u_\ell}_X^2,
\end{equation}
where
\(C_{\qo}(K)=\min\{4K^2,K^4\}(1+\vartheta_K)\).
In particular, we have \(\mathfrak C_{\cG}(N)=O(N^{\sigma_K})\), and
\(C_{\qo}(1)=1\) recovers the sharp orthogonal bound.
\end{theorem}

\begin{proof}
Let \(e_\ell:=(\Id-G_\ell)u=u-u_\ell\).  
Then for every \(j>\ell\), 
compatibility gives
\[
  \Delta_jG_\ell
  =
  (G_j-G_{j-1})G_\ell
  =
  0,
\]
and hence \(\Delta_je_\ell=\Delta_ju=u_j-u_{j-1}\),
where \(\Delta_j=G_j-G_{j-1}\).
This gives
\[
  \sum_{k=\ell}^{\ell+N-1}
  \norm{u_{k+1}-u_k}_X^2
  =
  \sum_{j\in I}\norm{\Delta_je_\ell}_X^2
  =
  \norm{\mathcal D_Ie_\ell}_{\ell^2(I;X)}^2 ,
\]
where \(I=\{\ell+1,\ldots,\ell+N\}\) and
\[
  \mathcal D_I:X\to\ell^2(I;X),
  \qquad
  \mathcal D_Ix:=(\Delta_jx)_{j\in I}.
\]
We will bound the norm of \(\mathcal D_I\) through its adjoint
\[
  \mathcal D_I^*:\ell^2(I;X)\rightarrow X,
  \qquad
  \mathcal D_I^*(z_j)_{j\in I}
  =
  \sum_{j\in I}\Delta_j^*z_j.
\]

The adjoint projections \((G_j^*)_{j\ge0}\) form a compatible projection
chain with the same uniform bound \(K\), and \(\Delta_j^*\) are its detail
projections.  
Hence \Cref{thm:upper-square} gives
\[
  \norm{\mathcal D_I^*z}_X^2
  =
  \norm{\sum_{j\in I}\Delta_j^*z_j}_X^2 
  \le
  (1+\vartheta_K)N^{\sigma_K}
  \sum_{j\in I}\norm{\Delta_j^*z_j}_X^2.
\]
By \eqref{eq:detail-bound}, we have \(\norm{\Delta_j}\le \chi=\min\{2K,K^2\}\), and thus
\[
  \norm{\mathcal D_I^*z}_{X}^2
  \le
  \chi^2(1+\vartheta_K)N^{\sigma_K}
  \sum_{j\in I}\norm{z_j}_X^2 ,
\]
establishing
\(\norm{\mathcal D_I}^2 = \norm{\mathcal D_I^*}^2 \le \chi^2(1+\vartheta_K)N^{\sigma_K}\).
\end{proof}

In particular, if \(C_a\le C_0\) and \(\gamma\ge\gamma_0>0\) uniformly over
a family of nested Galerkin problems, then \eqref{eq:main-gqo} holds uniformly
for the family with \(K=C_0/\gamma_0\).

\subsection{Consequences for adaptive methods}
\label{sec:afem}

Consider an adaptive algorithm of the form
\textsf{SOLVE}--\textsf{ESTIMATE}--\textsf{MARK}--\textsf{REFINE} that
generates nested discrete spaces satisfying \eqref{eq:discrete-infsup},
with Galerkin solutions \(u_\ell=G_\ell u\), and let \(\eta_\ell\ge0\)
be the a posteriori error estimator computed from \(u_\ell\).
In the standard conforming axiomatic setting, D\"orfler marking together
with stability, reduction, and discrete reliability of the estimator
yields estimator reduction, reliability, and quasi-monotonicity:
\begin{equation}\label{eq:estimator-axioms}
  \eta_{\ell+1}^2\le\kappa\eta_\ell^2+C_{\mathrm{est}}\norm{u_{\ell+1}-u_\ell}_X^2,
  \qquad
  \norm{u-u_\ell}_X\le C_{\mathrm{rel}}\eta_\ell,
  \qquad
  \eta_{\ell+j}^2\le C_{\mathrm{mon}}\eta_\ell^2,
\end{equation}
with \(0<\kappa<1\) and \(\ell,j\ge0\); see
\citet[Lemmas~3.4, 3.5, and~4.7]{CFPP2014} and
\citet[Section~5.3]{Feischl2022}.
\citet{Feischl2022} showed that \eqref{eq:estimator-axioms} and relaxed
GQO, \(\mathfrak C_{\cG}(N)=o(N)\), imply linear convergence of the
estimator.  Indeed, summing the estimator reduction over \(N\) steps,
using reliability, and averaging with quasi-monotonicity gives
\[
  \eta_{\ell+N}^2
  \le\frac{C_{\mathrm{mon}}}{N}\sum_{i=1}^N\eta_{\ell+i}^2
  \le\frac{C_{\mathrm{mon}}\bigl(\kappa+C_{\mathrm{est}}C_{\mathrm{rel}}^2
       \mathfrak C_{\cG}(N)\bigr)}{(1-\kappa)N}\,\eta_\ell^2 .
\]
The factor is less than \(1\) for large \(N\), so
\(\eta_{\ell+n}^2\le C_{\mathrm{lin}}q_{\mathrm{lin}}^n\eta_\ell^2\) for
all \(\ell,n\ge0\), with \(0<q_{\mathrm{lin}}<1\).  Under the remaining
hypotheses of \citet[Theorem~3]{Feischl2022}, rate optimality follows.

\Cref{thm:main-gqo} supplies the relaxed estimate with
\(\mathfrak C_{\cG}(N)\le C_{\qo}(K)N^{\sigma_K}\).  Its constants do
not depend on the nested sequence, so they apply to every sequence the
algorithm may generate.  
The resulting linear-convergence constants depend only on the estimator,
marking, and refinement constants and on \(C_a/\gamma\); the optimality
constants depend in addition on the approximation exponent.
In particular, \Cref{thm:main-gqo} gives an alternative proof of the
relaxed GQO estimate used in \citet{Feischl2022} for the Taylor--Hood
discretization of the Stokes problem in \(d=2,3\) and for the
Johnson--N\'ed\'elec FEM--BEM coupling, replacing the finite
block-\(LU\) argument.  The remaining estimator, marking, and
refinement arguments apply unchanged.

Linear convergence and reliability also give the estimator-controlled
square summability
\begin{equation}\label{eq:estimator-summability}
  \sum_{k=\ell}^{\infty}\norm{u_{k+1}-u_k}_X^2
  \le2C_{\mathrm{rel}}^2\sum_{k=\ell}^{\infty}
  \bigl(\eta_{k+1}^2+\eta_k^2\bigr)
  \lesssim\eta_\ell^2;
\end{equation}
see \citet[Proposition~4.11]{CFPP2014} and
\citet[Remark~20]{Feischl2022}.  This does not imply
\eqref{eq:intro-full-gqo}, since reliability gives no upper bound for
\(\eta_\ell\) in terms of \(\norm{u-u_\ell}_X\).

\section{Symmetric indefinite problems and finite-window sign geometry}
\label{sec:sign}

The projection proof above applies to general Petrov--Galerkin methods
and does not require symmetry. We now specialize to a self-adjoint
problem with identical trial and test spaces, and rephrase the argument in terms of the positive and negative spectral parts of the detail spaces.

The local spectral splitting below already appeared in
\citet{Gantumur2014v2}; the unresolved step there was the passage to a
uniformly complemented infinite positive--negative decomposition.
Here we retain the finite-dimensional construction and control its
finite-window assembly by \Cref{thm:two-sided-square}.

Let \(X\) be a real Hilbert space, let
\(T\in\cL(X)\) be a bounded, self-adjoint, and invertible operator, and define
\[
  a(x,y):=\ip{Tx}{y}_X
  \qquad (x,y\in X).
\]
Let \(X_0\subset X_1\subset\cdots\subset X\) be nested
finite-dimensional subspaces, not all zero,
and let
\(Q_\ell\) denote the orthogonal projector onto \(X_\ell\).
Assume
that the compressions
\[
  T_\ell:=Q_\ell T|_{X_\ell}:X_\ell\to X_\ell
\]
are uniformly invertible, with
\[
  \norm{T_\ell^{-1}}\le\beta_d^{-1}
  \qquad(\ell\ge0)
\]
for some \(\beta_d>0\).  
Let \(G_\ell:X\to X_\ell\) be the Galerkin projection characterized by
\[
  a(G_\ell x,v_\ell)=a(x,v_\ell)
  \qquad(x\in X,\ v_\ell\in X_\ell).
\]
Equivalently, \(T_\ell G_\ell=Q_\ell T\), so
\(G_\ell=T_\ell^{-1}Q_\ell T\) and
\[
  1\le K:=\sup_{\ell\ge0}\norm{G_\ell}
  \le\frac{\norm{T}}{\beta_d}<\infty.
\]
Set \(X_{-1}:=\{0\}\), \(G_{-1}:=0\), and define
\(E_j:=\ran(G_j-G_{j-1})\).
The compatibility of the
Galerkin projections implies that finite sums of distinct \(E_j\) are
algebraically direct.  
We also have
\[
  a(E_j,X_{j-1})=0,
\]
and hence symmetry gives
\[
  a(E_i,E_j)=0
  \qquad (i\ne j).
\]

Let $Q_{E_j}$ be the orthogonal projector onto $E_j$ and define the local
compression
\begin{equation}\label{eq:Bj}
 B_j:=Q_{E_j}T|_{E_j}:E_j\to E_j.
\end{equation}

\begin{lemma}[Local detail stability]\label{lem:Bj-inverse}
Each $B_j$ is self-adjoint and invertible.
We have
\begin{equation}\label{eq:Bj-inverse}
 B_j^{-1}b=(\Id-G_{j-1})T_j^{-1}b,
 \qquad
 b\in E_j ,
\end{equation}
and thus
\begin{equation}
\label{eq:mu-bound}
  \norm{B_j^{-1}}
  \le \frac{K}{\beta_d},
  \qquad
    \mu:=\inf_{E_j\ne\{0\}}\norm{B_j^{-1}}^{-1}\ge\frac{\beta_d}{K}.
\end{equation}
\end{lemma}

\begin{proof}
Self-adjointness follows immediately from the self-adjointness of \(T\).
Let \(b\in E_j\), set
\[
  z:=T_j^{-1}b\in X_j,
  \qquad
  y:=(\Id-G_{j-1})z.
\]
Since \(G_jz=z\), we have
\(y=(G_j-G_{j-1})z\in E_j\).
For \(v\in E_j\subset X_j\), the identity \(T_jz=b\) gives
\[
  \ip{Tz}{v}_X=\ip{b}{v}_X.
\]
Now note that \(G_{j-1}z\in X_{j-1}\), while
\(a(X_{j-1},E_j)=0\).  Hence we infer
\[
  \ip{Ty}{v}_X
  =
  \ip{Tz}{v}_X-\ip{TG_{j-1}z}{v}_X
  =
  \ip{b}{v}_X,
\]
yielding \(B_jy=b\).  Thus \(B_j\) is surjective and, since \(E_j\) is
finite-dimensional, invertible.  We have also proved
\eqref{eq:Bj-inverse}.  
Finally, we infer
\[
  \norm{B_j^{-1}}
  \le
  \norm{\Id-G_{j-1}}\norm{T_j^{-1}}
  \le
  \frac{K}{\beta_d},
\]
by \(\norm{\Id-G_{j-1}}\le K\),
which establishes \eqref{eq:mu-bound}.
\end{proof}

The operator
\(R_j=\sgn(B_j)\)
is a self-adjoint involution on \(E_j\), and
\begin{equation}
\label{eq:local-sign-coercivity}
  a(y,R_jy)
  =
  \ip{B_jy}{R_jy}_X
  =
  \ip{|B_j|y}{y}_X
  \ge
  \mu\norm{y}_X^2
  \qquad (y\in E_j).
\end{equation}
For a consecutive interval \(I\) of detail levels, set
\(E_I=\bigoplus_{j\in I}E_j\)
with the norm inherited from \(X\), and define the finite-window sign operator
\begin{equation}\label{eq:window-sign}
 \mathfrak S_I\Big(\sum_{j\in I}y_j\Big)
 :=\sum_{j\in I}R_jy_j.
\end{equation}
The algebraic directness of the detail spaces makes \(\mathfrak S_I\)
well defined, and
\(\mathfrak S_I^2=\Id_{E_I}\).
In general, however, \(\mathfrak S_I\) need not be self-adjoint in the
ambient Hilbert inner product, since the detail spaces \(E_j\) need not be
mutually orthogonal in that inner product.

\begin{proposition}[Sublinear finite-window sign separation]
\label{prop:sign-window}
Let \(I\) be a finite consecutive interval of detail levels and
\(n=|I|\). 
Then with \(\chi=\min\{2K,K^2\}\), we have
\begin{equation}
\label{eq:sign-window-bound}
  \norm{\mathfrak S_I}_{\cL(E_I)}
  \le
  \chi(1+\vartheta_K)n^{\sigma_K} .
\end{equation}
Moreover, for every \(y_j\in E_j\), \(j\in I\), and
\(w=\sum_{j\in I}y_j\), we have
\begin{equation}
\label{eq:sign-coercive-window}
  \sum_{j\in I}\norm{y_j}_X^2
  \le
  \frac{\norm{T}}{\mu}
  \norm{\mathfrak S_I}_{\cL(E_I)}
  \norm{w}_X^2 ,
\end{equation}
which implies that
\begin{equation}
\label{eq:sign-gqo}
  \sum_{j\in I}\norm{y_j}_X^2
  \le
  \frac{\norm{T}\chi(1+\vartheta_K)}{\mu}
  n^{\sigma_K}\norm{w}_X^2.
\end{equation}
\end{proposition}

\begin{proof}
Since each \(R_j\) is an isometry of \(E_j\), the two inequalities in
\Cref{thm:two-sided-square} give
\[
  \norm{\mathfrak S_Iw}_X^2
  \le
  A_Kn^{\sigma_K}\sum_{j\in I}\norm{R_jy_j}_X^2
  =
  A_Kn^{\sigma_K}\sum_{j\in I}\norm{y_j}_X^2
  \le
  A_KB_Kn^{2\sigma_K}\norm{w}_X^2.
\]
Since \(\sqrt{A_KB_K}=\chi A_K=\chi(1+\vartheta_K)\), the bound \eqref{eq:sign-window-bound} follows.

Cross-level \(a\)-orthogonality and
\eqref{eq:local-sign-coercivity} yield
\[
  \mu\sum_{j\in I}\norm{y_j}_X^2
  \le
  \sum_{j\in I}a(y_j,R_jy_j)
  =
  a(w,\mathfrak S_Iw)
  \le
  \norm{T}\norm{w}_X\norm{\mathfrak S_Iw}_X,
\]
which proves \eqref{eq:sign-coercive-window}.
Combining it with \eqref{eq:sign-window-bound} gives
\eqref{eq:sign-gqo}.
\end{proof}

For \(u_j:=G_ju\), take \(I=\{\ell+1,\ldots,\ell+N\}\) and \(y_j=u_j-u_{j-1}\). 
Since
\[
  w=u_{\ell+N}-u_\ell=G_{\ell+N}(u-u_\ell),
  \qquad
  \norm{w}_X\le K\norm{u-u_\ell}_X,
\]
the bound \eqref{eq:sign-gqo} yields relaxed GQO. 
This recovers \Cref{thm:main-gqo} in the self-adjoint case, with a different constant, through the spectral signs.

Full GQO corresponds to uniform bounds on the sign operators, as we
show in \Cref{prop:full-sign-characterization} below.
We first recall a classical Hilbert-space fact: a Schauder decomposition
is equivalent to an orthogonal one if and only if it satisfies a
two-sided square-function estimate; see, e.g.,
\citet[Chapter~VI, \S5]{GohbergKrein1969}.
For our purposes it is convenient to use an equivalent form in which the lower square-function bound is replaced by an upper square-function bound for the adjoint detail projections.

\begin{lemma}[Square functions and orthogonal similarity]
\label{lem:orthogonal-similarity}
Let \((P_j)_{j\ge0}\) be a compatible projection chain on a Hilbert space
\(H\), with \(P_j\to\Id\) strongly. 
Set \(\Delta_j=P_j-P_{j-1}\) with \(P_{-1}=0\).  If
\begin{equation}\label{eq:two-analysis-bounds}
  \sum_{j\ge0}\norm{\Delta_jx}^2\le C\norm{x}^2,
  \qquad
  \sum_{j\ge0}\norm{\Delta_j^*x}^2\le C_*\norm{x}^2,
\end{equation}
then \(W:=\sum_{j\ge0}\Delta_j^*\Delta_j\) converges strongly,
\(C_*^{-1}\Id\le W\le C\Id\), and \(P_k^*W=WP_k\).  Hence
\(W^{1/2}P_kW^{-1/2}\) is an orthogonal projection chain.  Conversely,
similarity to an orthogonal projection chain implies
\eqref{eq:two-analysis-bounds}.
\end{lemma}

\begin{proof}
First observe that the partial sums
\(W_N:=\sum_{j\le N}\Delta_j^*\Delta_j\)
satisfy \(0\le W_N\le W_{N+1}\le C\Id\).  
To see directly that \(W_N\)
converges strongly, let \(M>N\) and \(A:=W_M-W_N\).  Cauchy--Schwarz
for the positive semidefinite form \((Ax,y)\), together with
\(0\le A\le C\Id\), gives
\[
  \|Ax\|^2\le C(Ax,x).
\]
Since \((W_Nx,x)\) is increasing and bounded, it is Cauchy, 
and hence
\((W_Nx)\) is Cauchy for every \(x\).  
Thus we have \(W_N\to W\) strongly,
with \(0\le W\le C\Id\).

Since \(\Delta_j^2=\Delta_j\) and
\(\sum_{j=0}^N\Delta_j=P_N\to\Id\) strongly, we have
\[
  \norm{x}^2
  =\sum_{j\ge0}\ip{\Delta_j^*x}{\Delta_jx}
  \le
  \Big(\sum_{j\ge0}\norm{\Delta_j^*x}^2\Big)^{1/2}
  \Big(\sum_{j\ge0}\norm{\Delta_jx}^2\Big)^{1/2}
  \le
  \sqrt{C_*}\,\norm{x}\,\ip{Wx}{x}^{1/2}.
\]
Hence \(C_*^{-1}\Id\le W\le C\Id\), so
\((x,y)_W:=(Wx,y)\) is an equivalent Hilbert inner product.

Finally, compatibility gives
\[
  P_k^*W=\sum_{j\le k}\Delta_j^*\Delta_j=WP_k.
\]
Thus each \(P_k\) is self-adjoint in the \(W\)-inner product and hence orthogonal there. 
Equivalently,
\(W^{1/2}P_kW^{-1/2}\) is an orthogonal projection chain.
Conversely, similarity to an orthogonal projection chain transfers the
orthogonal square-function identity, and its adjoint counterpart, to
give \eqref{eq:two-analysis-bounds}.
\end{proof}

In the self-adjoint case the identity \(G_j^*T=TG_j\) supplies the
adjoint square-function bound from the primal one.  Thus
\Cref{lem:orthogonal-similarity} also gives a characterization of full
GQO by similarity to an orthogonal projection chain.  We now show that
full GQO is equivalently characterized by the local spectral signs.

\begin{proposition}[Full GQO and uniform sign bounds]
\label{prop:full-sign-characterization}
Under the assumptions of this section, full GQO, with a constant
uniform over all target vectors in \(X\), is equivalent to
\[
  \sup_I\norm{\mathfrak S_I}_{\cL(E_I)}<\infty,
\]
where the supremum is over finite consecutive intervals.
If \(\bigcup_\ell X_\ell\) is dense in \(X\), these conditions are also
equivalent to the existence of a boundedly invertible \(S\) and an
orthogonal projection chain \((\widehat Q_\ell)\) satisfying
\(G_\ell=S\widehat Q_\ell S^{-1}\).
\end{proposition}

\begin{proof}
Suppose first that full GQO holds, and set
\(\Delta_j:=G_j-G_{j-1}\). Taking the starting level to be zero and
adding the term \(\Delta_0=G_0\), we obtain a finite constant
\(C_{\rm an}\) such that
\begin{equation}\label{eq:full-analysis}
  \sum_{j\ge0}\norm{\Delta_jx}^2
  \le C_{\rm an}\norm{x}^2
  \qquad (x\in X).
\end{equation}
We first deduce uniform bounds for the sign operators. Let
\(w=\sum_{j\in I}y_j\), with \(y_j\in E_j\). Then
\(y_j=\Delta_jw\). Galerkin orthogonality and symmetry give
\(G_j^*T=TG_j\), hence \(\Delta_j^*T=T\Delta_j\), and therefore
\[
  \ip{T\mathfrak S_Iw}{z}
  =
  \sum_{j\in I}\ip{|B_j|y_j}{\Delta_jz}
  \qquad (z\in X).
\]
Since \(\norm{B_j}\le\norm{T}\), Cauchy--Schwarz and
\eqref{eq:full-analysis} yield
\(\abs{\ip{T\mathfrak S_Iw}{z}}
  \le
  \norm{T}C_{\rm an}\norm{w}\norm{z}\).
Thus we infer that
\(\norm{\mathfrak S_I}
  \le
  \norm{T^{-1}}\norm{T}C_{\rm an}\)
uniformly in \(I\).

Now suppose in addition that \(\bigcup_\ell X_\ell\) is dense in \(X\).
Then \(G_\ell\to\Id\) strongly. Indeed, for \(v\in X_m\) and
\(\ell\ge m\), \(G_\ell v=v\), so
\[
  \norm{G_\ell x-x}
  \le (K+1)\norm{x-v},
\]
and density gives the claim. Moreover, recall that
\(\Delta_j^*=T\Delta_jT^{-1}\), and hence
\[
  \sum_{j\ge0}\norm{\Delta_j^*x}^2
  \le
  \norm{T}^2\norm{T^{-1}}^2
  C_{\rm an}\norm{x}^2.
\]
Together with \eqref{eq:full-analysis}, this gives the asserted
similarity by \Cref{lem:orthogonal-similarity}.

Conversely, suppose that
\(\norm{\mathfrak S_I}\le M\) uniformly in \(I\). 
Then \eqref{eq:sign-coercive-window} gives
\[
  \sum_{j\in I}\norm{y_j}^2
  \le
  \frac{\norm{T}M}{\mu}
  \norm{\sum_{j\in I}y_j}^2 .
\]
Taking \(I=\{\ell+1,\ldots,\ell+N\}\) and
\(y_j=\Delta_j u\), we have
\(
  \sum_{j\in I}y_j
  =
  G_{\ell+N}(\Id-G_\ell)u\),
whose norm is at most \(K\norm{(\Id-G_\ell)u}\). This proves full GQO.

Finally, under the density assumption, similarity to an orthogonal
projection chain gives \eqref{eq:full-analysis} by
\Cref{lem:orthogonal-similarity}. Applying that estimate to
\((\Id-G_\ell)u\), and using
\(\Delta_j(\Id-G_\ell)=0\) for \(j\le\ell\), establishes full GQO.
\end{proof}

\begin{remark}
Factorizations give such similarities directly. In the setting of
\Cref{cor:finite-block-lu},
recall that we had
\(P_j=U^{-1}Q_jU\),
so the upper triangular factor \(U\) conjugates the chain \((P_j)\) to
the coordinate projections \((Q_j)\). Bounded exact upper triangular
factors with bounded inverses therefore give similarity to an
orthogonal chain, and hence full GQO by
\Cref{lem:orthogonal-similarity}. This is the idea behind
\citet[Lemma~3.3]{Feischl2019}.
\end{remark}

\section{Full GQO can fail under uniform inf-sup stability}
\label{sec:counterexample}

We adapt the power-weighted Fourier construction of
\citet[Proposition~3.10]{GarrigosWojtaszczyk2014} to produce a self-adjoint involution and a nested sequence of uniformly stable
Galerkin spaces for which full GQO fails. The same example has
finite-window constants growing like \(N^\alpha\), with \(0<\alpha<1\),
and an explicit target whose squared Galerkin increments have divergent
sum.

\subsection{A weighted trigonometric Schauder basis}
\label{subsec:trig}

Fix \(0<\alpha<1\) and set
\[
  w_\alpha(t)=|t|^\alpha,
  \qquad
  H_{\pm\alpha}=L^2(\Torus,w_\alpha^{\pm1}\,dt;\R),
\]
where \(\Torus=\R/\Z\) is represented by \((-\frac12,\frac12]\).
We will repeatedly use
\begin{equation}\label{eq:weight-integrals}
  \int_\Torus w_\alpha=\frac{2^{-\alpha}}{1+\alpha},
  \qquad
  \int_\Torus w_\alpha^{-1}=\frac{2^{\alpha}}{1-\alpha}.
\end{equation}
Let \((x_n)_{n\ge1}\) be the real trigonometric system in its natural
ordering, normalized in \(L^2(\Torus)\):
\[
x_1=1,\qquad
x_{2k}=\sqrt2\sin(2\pi kt),\qquad
x_{2k+1}=\sqrt2\cos(2\pi kt),
\qquad k\ge1.
\]
Both weights \(w_\alpha^{\pm1}\) belong to the Muckenhoupt class \(A_2\), 
and hence the trigonometric system is a Schauder
basis of \(H_\alpha\) and of \(H_{-\alpha}\) by \citet{HMW1973}.

Multiplication by \(w_\alpha^{-1}\) is an isometric isomorphism
\(H_{-\alpha}\to H_\alpha\).  Consequently,
\(z_n:=w_\alpha^{-1}x_n\) is a Schauder basis of \(H_\alpha\), and
\begin{equation}\label{eq:trig-biorthogonality}
  \ip{x_n}{z_m}_{H_\alpha}
  =
  \int_\Torus x_nx_m\,dt
  =
  \delta_{nm}.
\end{equation}
Thus the coordinate functional of \(x_n\) is
\(x_n^*(u)=\ip{u}{z_n}_{H_\alpha}=\int_\Torus u\,x_n\,dt\), the
ordinary Fourier coefficient of \(u\).  The coordinate functional of
\(z_n\) is \(z_n^*(v)=\ip{v}{x_n}_{H_\alpha}\).
The partial-sum projections
\[
S_N^xu=\sum_{n\le N}\ip{u}{z_n}_{H_\alpha}x_n
\qquad\text{and}\qquad
S_N^zv=\sum_{n\le N}\ip{v}{x_n}_{H_\alpha}z_n,
\] 
therefore satisfy \(S_N^z=(S_N^x)^*\) as operators in \(H_\alpha\).

Now set \(\mathcal H_\alpha=H_\alpha\oplus H_\alpha\).  Let
\(J(u,v)=(v,u)\), so that \(J=J^*=J^{-1}\), and define
\begin{equation}\label{eq:fn-basis}
  f_{2n-1}=2^{-1/2}(x_n,z_n),
  \qquad
  f_{2n}=2^{-1/2}(x_n,-z_n),
  \qquad
  n\ge1.
\end{equation}
A direct calculation using \eqref{eq:trig-biorthogonality} gives
\begin{equation}\label{eq:J-orthogonal}
  \ip{Jf_{2n-1}}{f_{2m-1}}=\delta_{nm},
  \qquad
  \ip{Jf_{2n}}{f_{2m}}=-\delta_{nm},
  \qquad
  \ip{Jf_{2n-1}}{f_{2m}}=0.
\end{equation}
Moreover, the squared norm 
\(L_n :=\norm{f_{2n-1}}^2=\norm{f_{2n}}^2=\frac12(\norm{x_n}_{H_\alpha}^2+\norm{z_n}_{H_\alpha}^2)\)
satisfies
\begin{equation}\label{eq:fn-norm}
  1\le L_n\le\int_\Torus w_\alpha+\int_\Torus w_\alpha^{-1}.
\end{equation}
Indeed, the upper bound follows from \(|x_n|^2\le2\),
and the lower bound is simply
\[
L_n\ge\norm{x_n}_{H_\alpha}\norm{z_n}_{H_\alpha}
\ge|\ip{x_n}{z_n}_{H_\alpha}|=1.
\]

By \eqref{eq:J-orthogonal}, the vectors \(f_j\) are mutually
\(J\)-orthogonal with \(\ip{Jf_j}{f_j}=\varepsilon_j:=(-1)^{j+1}\).
Hence the finite-rank operators
\begin{equation}\label{eq:partial-sum}
  \Pi_mU:=\sum_{j=1}^m\varepsilon_j\ip{JU}{f_j}f_j,
  \qquad
  \Pi_0:=0,
\end{equation}
satisfy \(\Pi_mf_j=f_j\) for \(j\le m\) and
\(\Pi_mf_j=0\) for \(j>m\).
For \(U=(u,v)\), a direct calculation gives
\begin{equation}\label{eq:even-odd}
  \Pi_{2N}U=(S_N^xu,S_N^zv),
  \qquad
  \Pi_{2N-1}U=\Pi_{2N-2}U+\ip{JU}{f_{2N-1}}f_{2N-1}.
\end{equation}
The rank-one term has norm \(\norm{f_{2N-1}}^2=L_N\).
Since \((x_n)\) is a Schauder basis of \(H_\alpha\), its partial-sum
projections \(S_N^x\) are uniformly bounded. Together with
\eqref{eq:fn-norm}, this gives
\begin{equation}\label{eq:basis-constant}
  K_\alpha:=\sup_{m\ge1}\norm{\Pi_m}
  \le
  \sup_{N\ge1}\norm{S_N^x}_{\cL(H_\alpha)}
  +\sup_{N\ge1}L_N
  <\infty ,
\end{equation}
where we have used \(S_N^z=(S_N^x)^*\).
The linear span of \((f_j)\) is dense, since
\[
(x_n,0)=2^{-1/2}(f_{2n-1}+f_{2n}),
\qquad
(0,z_n)=2^{-1/2}(f_{2n-1}-f_{2n}).
\]
Moreover, \(\Pi_mV=V\) for all sufficiently large \(m\) whenever
\(V\) belongs to this span. Uniform boundedness therefore gives
\(\Pi_m\to\Id\) strongly on \(\mathcal H_\alpha\). Thus
\((f_j)_{j\ge1}\) is a Schauder basis of \(\mathcal H_\alpha\), with
partial-sum projections \(\Pi_m\).

The rotation \eqref{eq:fn-basis} places the weighted-Fourier example
in the self-adjoint Galerkin setting: in the next subsection the spaces
\(\spann\{f_1,\ldots,f_m\}\) serve simultaneously as trial and test
spaces for the fixed self-adjoint involution \(J\).

\subsection{A uniformly stable Galerkin sequence}

Consider the symmetric indefinite bilinear form
\[
  a(U,V):=\ip{JU}{V}_{\mathcal H_\alpha},
\]
and the nested spaces \(X_0:=\{0\}\) and
\(X_m:=\spann\{f_1,\ldots,f_m\}\) for \(m\ge1\).
By \eqref{eq:J-orthogonal}, the Gram matrix of \(a\) on \(X_m\) in the
basis \(f_1,\ldots,f_m\) is \(\operatorname{diag}(\varepsilon_1,\ldots,\varepsilon_m)\).
In particular, the Galerkin problem on \(X_m\) is uniquely solvable.
For \(U\in\mathcal H_\alpha\) and \(k\le m\), by \eqref{eq:partial-sum} we get
\[
  a(\Pi_mU,f_k)
  =\sum_{j=1}^m\varepsilon_j\ip{JU}{f_j}\ip{Jf_j}{f_k}
  =\varepsilon_k^2\ip{JU}{f_k}
  =a(U,f_k).
\]
Hence the partial-sum projection \(\Pi_m\) is the Galerkin
projection onto \(X_m\).  We write
\(\mathcal G_\alpha:=(\Pi_m)_{m\ge0}\).

\begin{lemma}[Exact discrete stability constant]
\label{lem:counter-infsup}
We have
\begin{equation}\label{eq:counter-beta}
  \beta_m
  :=
  \inf_{0\ne U_m\in X_m}
  \sup_{0\ne V_m\in X_m}
  \frac{\abs{\ip{JU_m}{V_m}}}
       {\norm{U_m}\norm{V_m}}
  =\frac1{\norm{\Pi_m}} ,
\end{equation}
for \(m\ge1\).
In particular, this yields
\begin{equation}\label{eq:counter-uniform-infsup}
  \inf_{m\ge1}\beta_m=K_\alpha^{-1}>0.
\end{equation}
\end{lemma}

\begin{proof}
Let \(Q_m\) be the orthogonal projection onto \(X_m\), and set
\[
  A_m:=Q_mJ|_{X_m}:X_m\to X_m .
\]
The diagonal Gram matrix above shows that \(A_m\) is invertible. For
\(U_m\in X_m\),
\[
  \sup_{0\ne V_m\in X_m}
  \frac{\abs{\ip{JU_m}{V_m}}}{\norm{V_m}}
  =
  \norm{Q_mJU_m}
  =
  \norm{A_mU_m},
\]
and hence
\(\beta_m=\norm{A_m^{-1}}^{-1}\).

The Galerkin characterization of \(\Pi_m\) yields
\(\Pi_m=A_m^{-1}Q_mJ\).
Since \(Q_mJ\) is a contraction, we have \(\norm{\Pi_m}\le\norm{A_m^{-1}}\).
On the other hand, for \(Y\in X_m\), one has
\[
  Q_mJ(JY)=Y,\qquad \norm{JY}=\norm{Y},
\]
and so
\[
  \norm{A_m^{-1}Y}
  =\norm{\Pi_m(JY)}
  \le\norm{\Pi_m}\norm{Y}.
\]
Thus \(\norm{\Pi_m}=\norm{A_m^{-1}}\), which proves
\eqref{eq:counter-beta}. Taking the infimum over \(m\) yields
\eqref{eq:counter-uniform-infsup}.
\end{proof}

In the notation of \Cref{sec:sign}, take \(T=J\) and \(T_m=A_m\).
\Cref{lem:counter-infsup} gives \(\norm{T_m^{-1}}=\norm{\Pi_m}\le K_\alpha\),
so the assumptions there hold with \(\beta_d=K_\alpha^{-1}\).

\begin{lemma}[Size of the projection bound]
\label{lem:Kalpha-size}
There are absolute constants \(c,C>0\) such that
\begin{equation}\label{eq:Kalpha-size}
  \frac{c}{1-\alpha}\le K_\alpha\le\frac{C}{1-\alpha}
  \qquad(0<\alpha<1).
\end{equation}
\end{lemma}

\begin{proof}
For the lower bound, \eqref{eq:partial-sum} gives
\(\Pi_1=\ip{J\,\cdot\,}{f_1}f_1\). Since \(J\) is an isometry,
\eqref{eq:fn-norm} and \eqref{eq:weight-integrals} yield
\[
  K_\alpha\ge\norm{\Pi_1}=\norm{f_1}^2=L_1
  =\frac12\left(\int_\Torus w_\alpha+\int_\Torus w_\alpha^{-1}\right)
  \ge\frac1{2(1-\alpha)}.
\]
For the upper bound, \eqref{eq:fn-norm} and
\eqref{eq:weight-integrals} give
\(\sup_NL_N\lesssim(1-\alpha)^{-1}\). By
\eqref{eq:basis-constant}, it remains to show that
\begin{equation}\label{e:aim-SN-bound}
  \sup_{N\ge1}\norm{S_N^x}_{\cL(H_\alpha)}
  \lesssim\frac1{1-\alpha}.
\end{equation}

\emph{Step 1: the \(A_2\) constant.}
Write \(w_\alpha=\rho^\alpha\), where
\(\rho(t)=\operatorname{dist}_\Torus(t,0)\). For an arc \(I\) of
length \(h\), let \(r=\inf_I\rho\). If \(r\ge h\), then we have
\(r\le\rho\le2r\) on \(I\), and so
\[
  \left(\frac1{|I|}\int_Iw_\alpha\right)
  \left(\frac1{|I|}\int_Iw_\alpha^{-1}\right)
  \le (2r)^\alpha r^{-\alpha}
  =2^\alpha
  \le2.
\]
If \(r<h\), then we have \(I\subset\{\rho<2h\}\), which yields
\[
  \int_I w_\alpha^{\pm1}
  \le
  2\int_0^{2h}t^{\pm\alpha}\,dt
  =
  \frac{2(2h)^{1\pm\alpha}}{1\pm\alpha}.
\]
The powers of \(2h\) cancel in the product of averages, and hence
\([w_\alpha]_{A_2(\Torus)}
  \le{16}/(1-\alpha^2)\).
The linear \(A_2\) bound for the periodic Hilbert transform
\citep{Petermichl2007,Berkson2021} therefore gives
\[
  \norm{\mathcal H_{\mathrm{per}}}_{\cL(H_\alpha)}
  \lesssim [w_\alpha]_{A_2(\Torus)}
  \lesssim\frac1{1-\alpha}.
\]

\emph{Step 2: complete partial sums.}
Work in the complexification of \(H_\alpha\), which leaves the norms of real operators unchanged. 
With the convention
\[
  \widehat{\mathcal H_{\mathrm{per}}f}(k)
  =-i\operatorname{sgn}(k)\widehat f(k),
\]
the operator
\[
  P_+:=\tfrac12(\Id+i\mathcal H_{\mathrm{per}})
       +\tfrac12P_{\mathrm c},
  \qquad
  P_{\mathrm c}f:=\int_\Torus f,
\]
keeps exactly the frequencies \(k\ge0\).
By \eqref{eq:weight-integrals}, one has
\(\norm{P_{\mathrm c}}^2=(1-\alpha^2)^{-1}\), so
\(\norm{P_+}\lesssim(1-\alpha)^{-1}\).
Since the modulations \(M_mf=e^{2\pi imt}f\) are isometries of
\(H_\alpha\), the operator
\[
  P_{\le n}
  :=
  M_{-n}P_+M_n-M_{n+1}P_+M_{-n-1}
\]
is the projection onto the frequencies \(|k|\le n\), and
\(\norm{P_{\le n}}\lesssim(1-\alpha)^{-1}\).

\emph{Step 3: real partial sums.}
On real functions, we have \(P_{\le n}=S_{2n+1}^x\), so Step~2 gives the
desired bound for partial sums ending with a complete sine--cosine
pair. For an incomplete pair, we have
\[
  S_{2n}^x
  =
  S_{2n+1}^x-x_{2n+1}^*(\cdot)\,x_{2n+1}.
\]
Since \(|x_{2n+1}|^2\le2\), the rank-one term has norm at most
\(\norm{x_{2n+1}}_{H_\alpha}
  \norm{z_{2n+1}}_{H_\alpha}
  \le
  2/\sqrt{1-\alpha^2}\).
Hence \eqref{e:aim-SN-bound} holds for every \(N\), and
\eqref{eq:basis-constant} proves the upper bound.
\end{proof}

\subsection{Finite targets and sharp window growth}

The window constants of \(\mathcal G_\alpha\) reduce to a question
about trigonometric polynomials. Up to an orthogonal change within each
pair \((f_{2m-1},f_{2m})\), the coefficients of a vector in the basis
\((f_j)\) are unweighted Fourier coefficients, whereas the two
components of the norm carry the weights \(|t|^\alpha\) and
\(|t|^{-\alpha}\). The latter weight is bounded below on \(\Torus\),
so it controls the ordinary \(L^2\)-norm directly. The only possible
loss comes from the first component, where \(|t|^\alpha\) vanishes at
the origin. We therefore need to control how much \(L^2\)-mass a
trigonometric polynomial with few terms can concentrate near \(t=0\).
The following lemma gives matching upper and lower bounds.

\begin{lemma}[Concentration of sparse trigonometric polynomials]
\label{lem:sparse-concentration}
\leavevmode
\begin{enumerate}
\item[(a)] If \(p\) is a linear combination of at most \(n\) of the
functions \(x_j\), then we have
\[
  \norm{p}_{H_\alpha}^2\ge\tfrac12(8n)^{-\alpha}\norm{p}_{L^2(\Torus)}^2.
\]
\item[(b)] For \(N\ge2\), the \(L^2\)-normalized Fej\'er kernel
\(q_N:=F_N/\norm{F_N}_{L^2(\Torus)}\), where
\[
  F_N(t)=\sum_{|k|<N}\Bigl(1-\frac{|k|}{N}\Bigr)e^{2\pi ikt}
  =\frac1N\Bigl(\frac{\sin(\pi Nt)}{\sin(\pi t)}\Bigr)^2,
\]
lies in
\(\spann\{x_1,\ldots,x_{2N-1}\}\)
and satisfies
\[
  \norm{q_N}_{H_\alpha}^2\le \frac{\pi^3}{4} N^{-\alpha} .
\]
\end{enumerate}
\end{lemma}

\begin{proof}
(a) Write \(p=\sum_{j\in S}a_jx_j\) with \(|S|\le n\).
Since \(|x_j|\le\sqrt2\), the Cauchy--Schwarz inequality gives
\[
  \norm{p}_{L^\infty}^2\le2\Bigl(\sum_{j\in S}|a_j|\Bigr)^2\le2n\norm{p}_{L^2}^2.
\]
Let \(\delta=(8n)^{-1}\).  The set \(\{|t|<\delta\}\) has measure \(2\delta\), and so we have
\[
  \int_{|t|<\delta}p^2\le2\delta\norm{p}_{L^\infty}^2\le\tfrac12\norm{p}_{L^2}^2.
\]
Since \(w_\alpha\ge\delta^\alpha\) on the complement,
\(\norm{p}_{H_\alpha}^2\ge\delta^\alpha\int_{|t|\ge\delta}p^2
\ge\frac12\delta^\alpha\norm{p}_{L^2}^2\).
Note that the bound depends on the number of terms of \(p\), not on
their frequencies.

(b) The kernel \(F_N\) is real and even.
Hence it is a combination of
\(x_1,x_3,\ldots,x_{2N-1}\), which proves the first claim.
By Parseval, we have
\[
  \norm{F_N}_{L^2}^2
  =\sum_{|k|<N}\Bigl(1-\frac{|k|}N\Bigr)^2
  =\frac{2N^2+1}{3N}
  \ge\frac{2N}3.
\]
On \((-\frac12,\frac12]\) we have
\(|\sin\pi t|\ge2|t|\) and
\(|\sin\pi Nt|\le\min\{1,\pi N|t|\}\). Hence
\[
  F_N(t)
  \le
  \frac{\pi^2N}{4}
  \min\left\{1,\frac1{(\pi Nt)^2}\right\}.
\]
Using symmetry and splitting at \(t=1/(\pi N)\), we obtain
\[
\begin{aligned}
  \norm{F_N}_{H_\alpha}^2
  &\le
  \frac{\pi^4N^2}{8}
  \Big(
    \int_0^{1/(\pi N)}t^\alpha\,dt
    +
    \frac1{\pi^4N^4}
    \int_{1/(\pi N)}^{1/2}t^{\alpha-4}\,dt
  \Big) 
  \le
  \frac{\pi^{3-\alpha}}
       {2(1+\alpha)(3-\alpha)}
  N^{1-\alpha}.
\end{aligned}
\]
Together with
\(\norm{F_N}_{L^2}^2\ge2N/3\), this gives
\[
  \norm{q_N}_{H_\alpha}^2
  \le
  \frac{3\pi^{3-\alpha}}
       {4(1+\alpha)(3-\alpha)}
  N^{-\alpha}
  \le
  \frac{\pi^3}{4}N^{-\alpha}.
\]
\end{proof}

Part (b) shows that the scale \(1/n\) in part (a) is optimal.
The Fej\'er kernel has about \(2N\) terms and is essentially a bump of
height \(N\) and width \(1/N\) at the origin.

\begin{proposition}[Sharp window growth]
\label{prop:sharp-window-growth}
There exist absolute constants \(c,C>0\) such that
\begin{equation}\label{eq:sharp-window-growth}
  cN^\alpha
  \le\mathfrak C_{\mathcal G_\alpha}(N)
  \le \frac{C}{(1-\alpha)^3}N^\alpha
  \qquad(N\ge1) ,
\end{equation}
for every \(0<\alpha<1\).
\end{proposition}

\begin{proof}
Fix \(U\in\mathcal H_\alpha\), \(\ell\ge0\), and \(N\ge1\), and write
\(c_j:=\varepsilon_j\ip{JU}{f_j}\). For
\(I=\{\ell+1,\ldots,\ell+N\}\), \eqref{eq:partial-sum} and
\eqref{eq:fn-norm} give
\begin{equation}\label{eq:window-coefficients}
  \sum_{j\in I}|c_j|^2
  \le
  \sum_{k=\ell}^{\ell+N-1}
  \norm{\Pi_{k+1}U-\Pi_kU}^2
  \le
  L_\alpha^*\sum_{j\in I}|c_j|^2 ,
\end{equation}
where
\[
  L_\alpha^*
  :=
  \int_\Torus w_\alpha+\int_\Torus w_\alpha^{-1}
  =
  \frac{2^{-\alpha}}{1+\alpha}
  +\frac{2^\alpha}{1-\alpha}
  \lesssim \frac1{1-\alpha}.
\]

For the upper bound, set
\[
  V:=\sum_{j\in I}c_jf_j
  =\Pi_{\ell+N}(\Id-\Pi_\ell)U,
  \qquad
  \norm V\le K_\alpha\norm{(\Id-\Pi_\ell)U}.
\]
Extending the coefficients in \(I\) by zero, write
\[
  V=(p,w_\alpha^{-1}s),\qquad
  p=\sum_ma_mx_m,\quad s=\sum_mb_mx_m,
\]
where
\(a_m=({c_{2m-1}+c_{2m}})/{\sqrt2}\)
and
\(b_m=({c_{2m-1}-c_{2m}})/{\sqrt2}\).
The window meets at most \(N\) pairs, so both \(p\) and \(s\) contain
at most \(N\) of the functions \(x_m\), while orthogonality of the
change of coefficients gives
\[
  \sum_{j\in I}|c_j|^2
  =\norm p_{L^2}^2+\norm s_{L^2}^2.
\]
Since \(w_\alpha^{-1}\ge1\), and
\Cref{lem:sparse-concentration}(a) applies to \(p\), we have
\(\norm V^2
  \ge
  \tfrac12(8N)^{-\alpha}\norm p_{L^2}^2
  +\norm s_{L^2}^2\),
giving
\[
  \sum_{j\in I}|c_j|^2
  \le
  2\cdot8^\alpha K_\alpha^2N^\alpha
  \norm{(\Id-\Pi_\ell)U}^2.
\]
Together with \eqref{eq:window-coefficients}, this yields
\[
  \mathfrak C_{\mathcal G_\alpha}(N)
  \le
  2\cdot8^\alpha L_\alpha^*K_\alpha^2N^\alpha
  \lesssim
  \frac{N^\alpha}{(1-\alpha)^3},
\]
where we used \Cref{lem:Kalpha-size}.

For the lower bound, the case \(1\le N\le5\) follows by taking
\(U=f_1\) and \(\ell=0\), since the resulting window ratio is \(1\)
and \(N^\alpha\le5\). Suppose \(N\ge6\), and set
\(r:=\lfloor\frac{N+2}{4}\rfloor\).
Then \(r\ge2\), \(4r-2\le N\), and \(r\ge N/6\).
By \Cref{lem:sparse-concentration}(b), write
\[
  q_r=\sum_{m=1}^{2r-1}a_mx_m,
  \qquad
  \sum_m|a_m|^2=1.
\]
Since \(q_r\) is even, we have \(a_{2k}=0\) for \(1\le k<r\). For
\(U:=(q_r,0)\), \eqref{eq:fn-basis} gives
\[
  U
  =
  \sum_{m=1}^{2r-1}
  \frac{a_m}{\sqrt2}(f_{2m-1}+f_{2m})
  \in X_{4r-2}.
\]
Thus the first \(N\) increments contain every nonzero coefficient of
\(U\), and \eqref{eq:window-coefficients} gives a numerator at least
one. The proof of \Cref{lem:sparse-concentration}(b) gives
\(\norm U^2
  =\norm{q_r}_{H_\alpha}^2
  \le D r^{-\alpha}\)
with an absolute constant \(D\).
Therefore we have
\[
  \mathfrak C_{\mathcal G_\alpha}(N)
  \ge D^{-1}r^\alpha
  \ge \frac{1}{6D}N^\alpha.
\]
Together with the case \(N\le5\), this proves the lower bound with an
absolute constant \(c>0\).
\end{proof}

Thus no exponent below \(\alpha\) works for this fixed hierarchy.
Since \Cref{thm:main-gqo} applies with \(K=K_\alpha\) by
\Cref{lem:counter-infsup}, we have
\(\alpha\le\sigma_{K_\alpha}\).
As \(\alpha\uparrow1\), \Cref{lem:Kalpha-size,lem:counter-infsup} show
that \(K_\alpha\to\infty\) and \(\inf_m\beta_m\to0\). Thus exponents
approaching \(1\) require the projection bound to grow. To quantify
this, let
\[
  \sigma_{\mathrm{univ}}(K)
  :=
  \inf\Bigl\{
    \sigma\ge0:
    \exists\,C_K<\infty\ \text{such that}\
    \mathfrak C_{\cG}(N)\le C_KN^\sigma
  \Bigr\},
  \qquad K\ge1,
\]
where the estimate is required for every \(N\ge1\) and every compatible
projection chain \(\cG=(G_\ell)\) on
\(\ell^2(\mathbb N;\mathbb R)\) satisfying
\(\sup_\ell\norm{G_\ell}\le K\).
This entails no loss for separable real Hilbert spaces, since every
infinite-dimensional such space is isometrically isomorphic to
\(\ell^2\), while finite-dimensional chains may be embedded into
\(\ell^2\).

\begin{corollary}[Range of the universal exponent]
\label{cor:sigma-univ}
The function \(\sigma_{\mathrm{univ}}\) is nondecreasing, satisfies
\(\sigma_{\mathrm{univ}}(K)\le\sigma_K\), and there are absolute
constants \(c,C>0\) such that
\[
  cK^{-2}\le1-\sigma_{\mathrm{univ}}(K)\le CK^{-1}
\]
for all sufficiently large \(K\).
\end{corollary}

\begin{proof}
Monotonicity follows because the admissible class of chains increases
with \(K\). The projection estimate underlying
\Cref{thm:main-gqo} gives
\(\sigma_{\mathrm{univ}}(K)\le\sigma_K\). Since
\[
  1-\sigma_K
  =\frac{1}{4\ln2}K^{-2}+O(K^{-4}),
\]
this proves the lower bound.

For the upper bound, \Cref{prop:sharp-window-growth} gives
\(\sigma_{\mathrm{univ}}(K_\alpha)\ge\alpha\), while
\Cref{lem:Kalpha-size} gives
\(K_\alpha\le C_*/(1-\alpha)\) for an absolute constant \(C_*\).
For \(K>C_*\), take \(\alpha=1-C_*/K\). Then \(K_\alpha\le K\), so
monotonicity yields
\[
  \sigma_{\mathrm{univ}}(K)
  \ge\sigma_{\mathrm{univ}}(K_\alpha)
  \ge\alpha
  =1-{C_*}/{K}.\qedhere
\]
\end{proof}

\subsection{One fixed target for which full GQO fails}

The lower bounds in the preceding subsection use targets depending on
the window length. We now show that the same hierarchy fails full
GQO for a single fixed target.

\begin{theorem}[Failure of full GQO]
\label{thm:full-failure}
For \(g(t):=|t|^{-1/2}\), the target
\(U:=(g,0)\in\mathcal H_\alpha\) satisfies
\begin{equation}\label{eq:full-failure}
  \sum_{m=0}^{\infty}
  \norm{\Pi_{m+1}U-\Pi_mU}^2=\infty.
\end{equation}
In particular, we have \(\mathfrak C_{\mathcal G_\alpha}(\infty)=\infty\).
\end{theorem}

\begin{proof}
Since \(0<\alpha<1\),
we have \(g\in H_\alpha\) but \(g\notin L^2(\Torus)\).
Write \(g=\sum_{n\ge1}a_nx_n\) in \(H_\alpha\).
If \((a_n)\in\ell^2\), then the partial sums would converge in
\(L^2(\Torus)\).
Since \(w_\alpha\in L^\infty(\Torus)\), the inclusion
\(L^2(\Torus)\hookrightarrow H_\alpha\) is continuous, so the same
partial sums would converge in \(H_\alpha\) to their \(L^2\)-limit.
By uniqueness of the \(H_\alpha\)-limit, this limit must be \(g\),
contradicting \(g\notin L^2(\Torus)\). 
Thus we infer
\(\sum_{n}|a_n|^2=\infty\).

By \eqref{eq:partial-sum}, the coefficients of \(U\) are
\(c_j=\varepsilon_j\ip{JU}{f_j}\). Since \(JU=(0,g)\) and
\(a_n=\ip{g}{z_n}_{H_\alpha}\), 
we have \(c_{2n-1}=c_{2n}={a_n}/{\sqrt2}\) by \eqref{eq:fn-basis}.
Therefore, \eqref{eq:fn-norm} yields
\[
  \sum_{m=0}^{\infty}\norm{\Pi_{m+1}U-\Pi_mU}^2
  =\sum_{n=1}^{\infty}|a_n|^2L_n
  \ge\sum_{n=1}^{\infty}|a_n|^2
  =\infty.\qedhere
\]
\end{proof}

We conclude by relating the growth of the window constants to the sign operators of \Cref{sec:sign}.

\begin{remark}[Sign-operator interpretation]
\label{rem:counter-sign}
Each detail space
\(E_j=\spann\{f_j\}\) is one-dimensional, and
\eqref{eq:J-orthogonal} gives \(R_j=\varepsilon_j\). Hence, for a
finite sum \(U=\sum_jc_jf_j=(u,v)\), we have
\[
  \mathfrak S U=\sum_j\varepsilon_jc_jf_j
  =(w_\alpha v,w_\alpha^{-1}u),
\]
and
\[
  \ip{J\mathfrak SU}{U}_{\mathcal H_\alpha}
  =\sum_j|c_j|^2
  =\norm{u}_{L^2(\Torus)}^2
   +\norm{w_\alpha v}_{L^2(\Torus)}^2.
\]
Thus the coefficient square-sum norm is not equivalent to the ambient
norm of \(\mathcal H_\alpha\).

Indeed, let \(U_N=(q_N,0)\), where \(q_N\) is the normalized Fej\'er
kernel of \Cref{lem:sparse-concentration}(b), and set
\(I_N=\{1,\ldots,4N-2\}\). 
Then we have \(U_N\in E_{I_N}\) and
\[
  \ip{J\mathfrak S_{I_N}U_N}{U_N}=1,
  \qquad
  \norm{U_N}^2\le D N^{-\alpha}.
\]
Since \(J\) is an isometry, we infer
\(1
  \le
  \norm{\mathfrak S_{I_N}U_N}\norm{U_N}
  \le
  \norm{\mathfrak S_{I_N}}\norm{U_N}^2\),
and therefore
\[
  \norm{\mathfrak S_{I_N}}
  \ge D^{-1}N^\alpha.
\]
In particular,
\(\sup_I\norm{\mathfrak S_I}=\infty\), as predicted by
\Cref{prop:full-sign-characterization}.
\end{remark}

\section{Concluding remarks}
\label{sec:conc}

The weighted-Fourier example leaves open how much stronger than relaxed
GQO can be expected for locally refined finite element hierarchies,
such as those arising from Taylor--Hood discretizations. Local support
of the basis functions does not by itself control the detail spaces
\(E_j=(\Id-G_{j-1})X_j\), which involve a global projection.
For Taylor--Hood discretizations, \citet{Feischl2019} obtained the
error-perturbed GQO needed for optimality by exploiting wavelet-type
multilevel structure and off-diagonal decay under a mesh-grading
condition, whereas \citet{Feischl2022} showed that no such grading is
needed for relaxed GQO. This leaves the more structural question of
which properties of a finite element hierarchy (possibly including
mesh grading itself) govern stronger forms of quasi-orthogonality,
and which are merely sufficient or genuinely necessary. In the
self-adjoint setting, \Cref{prop:full-sign-characterization} recasts
unperturbed full GQO as uniform boundedness of the finite-window sign
operators.

At the purely projection-theoretic level,
\Cref{cor:sigma-univ} leaves a gap between \(K^{-2}\) and \(K^{-1}\)
for \(1-\sigma_{\mathrm{univ}}(K)\). The dyadic recursion in
\Cref{thm:upper-square} uses the worst admissible consecutive-block
angle at every generation, and it is not clear whether a single
compatible chain can remain near-extremal at all scales. Thus either
the general estimate can be improved or a more extreme chain is
needed to sharpen the lower bound. Determining the sharp dependence
of the universal exponent on \(K\) remains open.

A further natural direction is to pass beyond Hilbert spaces.
The classical coefficient estimates of
\citet{GurariiGurarii1971,James1972} yield qualitative analogues of the
finite-window bounds for uniformly bounded compatible projection
chains on superreflexive Banach spaces, with suitable powers replacing
squared norms. Obtaining useful quantitative dependence on the
projection bound and on the geometry of the ambient Banach space, and
developing applications to adaptive methods measured in non-Hilbert
norms, remain open problems.

Finally, \Cref{thm:main-gqo} applies directly to mixed formulations
whose discrete trial and test spaces are nested and uniformly inf-sup
stable. In finite element exterior calculus with nontrivial
cohomology, the moving discrete harmonic spaces destroy this
nestedness in the standard formulation
\citep{DemlowHirani,LiFEEC}. Extending the projection approach to such
moving constraint or kernel spaces, and more generally to nonnested
mixed discretizations, is another natural direction.

\section*{Acknowledgements}

I am especially grateful to Kris van der Zee for many discussions of this
problem, particularly during an Oberwolfach workshop.  Our repeated
attempts to repair the argument in \citet{Gantumur2014v2}, and his
persistence in returning to the problem, strongly influenced my view of
it.  This research was supported in part by the Natural Sciences and
Engineering Research Council of Canada (NSERC) through a Discovery Grant.

\bibliographystyle{abbrvnat}
\bibliography{gqo}

\end{document}